\documentclass[a4paper,11pt]{article}

\usepackage{makeidx}         
\usepackage{graphicx}        
\usepackage{psfrag}          
\usepackage{multicol}        
\usepackage{epstopdf}
\usepackage{amsmath}
\usepackage{amssymb}
\usepackage{amsfonts}
\usepackage{amsthm}
\usepackage{latexsym}
\usepackage[T1]{fontenc}
\usepackage{subfigure}       
\usepackage{float}
\usepackage{mathrsfs}
\usepackage{mathptmx}        
\usepackage{xspace}

\usepackage{tikz}
\usetikzlibrary{patterns}
\newlength{\hatchspread}
\newlength{\hatchthickness}
\newlength{\hatchshift}
\newcommand{\hatchcolor}{}
\tikzset{hatchspread/.code={\setlength{\hatchspread}{#1}},
         hatchthickness/.code={\setlength{\hatchthickness}{#1}},
         hatchshift/.code={\setlength{\hatchshift}{#1}},
         hatchcolor/.code={\renewcommand{\hatchcolor}{#1}}}
\tikzset{hatchspread=3pt,
         hatchthickness=0.4pt,
         hatchshift=0pt,
         hatchcolor=black}
\pgfdeclarepatternformonly[\hatchspread,\hatchthickness,\hatchshift,\hatchcolor]
   {custom north west lines}
   {\pgfqpoint{\dimexpr-2\hatchthickness}{\dimexpr-2\hatchthickness}}
   {\pgfqpoint{\dimexpr\hatchspread+2\hatchthickness}{\dimexpr\hatchspread+2\hatchthickness}}
   {\pgfqpoint{\dimexpr\hatchspread}{\dimexpr\hatchspread}}
   {
    \pgfsetlinewidth{\hatchthickness}
    \pgfpathmoveto{\pgfqpoint{0pt}{\dimexpr\hatchspread+\hatchshift}}
    \pgfpathlineto{\pgfqpoint{\dimexpr\hatchspread+0.15pt+\hatchshift}{-0.15pt}}
    \ifdim \hatchshift > 0pt
      \pgfpathmoveto{\pgfqpoint{0pt}{\hatchshift}}
      \pgfpathlineto{\pgfqpoint{\dimexpr0.15pt+\hatchshift}{-0.15pt}}
    \fi
    \pgfsetstrokecolor{\hatchcolor}
    \pgfusepath{stroke}
   }

\pgfdeclarepatternformonly[\hatchspread,\hatchthickness,\hatchshift,\hatchcolor]
   {custom north east lines}
   {\pgfqpoint{\dimexpr-2\hatchthickness}{\dimexpr-2\hatchthickness}}
   {\pgfqpoint{\dimexpr\hatchspread+2\hatchthickness}{\dimexpr\hatchspread+2\hatchthickness}}
   {\pgfqpoint{\dimexpr\hatchspread}{\dimexpr\hatchspread}}
   {
    \pgfsetlinewidth{\hatchthickness}
    \pgfpathmoveto{\pgfqpoint{\dimexpr\hatchshift-0.15pt}{-0.15pt}}
    \pgfpathlineto{\pgfqpoint{\dimexpr\hatchspread+0.15pt}{\dimexpr\hatchspread-\hatchshift+0.15pt}}
    \ifdim \hatchshift > 0pt
      \pgfpathmoveto{\pgfqpoint{-0.15pt}{\dimexpr\hatchspread-\hatchshift-0.15pt}}
      \pgfpathlineto{\pgfqpoint{\dimexpr\hatchshift+0.15pt}{\dimexpr\hatchspread+0.15pt}}
    \fi
    \pgfsetstrokecolor{\hatchcolor}
    \pgfusepath{stroke}
   }

\newcommand{\bP}{\mathbb{P}}
\newcommand{\bN}{\mathbb{N}}

\newcommand{\cP}{\mathcal{P}}
\newcommand{\cX}{\mathcal{X}}
\newcommand{\cD}{\mathcal{D}}

\newcommand{\cS}{\mathcal{S}}
\newcommand{\cT}{\mathcal{T}}

\newcommand{\msa}{{\sf a}}
\newcommand{\msb}{{\sf b}}
\newcommand{\msc}{{\sf c}}
\newcommand{\mse}{{\sf e}}
\newcommand{\msf}{{\sf f}}
\newcommand{\msp}{{\sf p}}
\newcommand{\msq}{{\sf q}}

\newcommand{\mss}{{\sf s}}

\newcommand{\msw}{{\sf w}}

\newcommand{\mC}{{\sf C}}
\newcommand{\mD}{{\sf D}}

\newcommand{\mK}{{\sf K}}
\newcommand{\mM}{{\sf M}}
\newcommand{\mN}{{\sf N}}
\newcommand{\mT}{{\sf T}}
\newcommand{\mV}{{\sf V}}
\newcommand{\mO}{{\sf O}}
\newcommand{\mS}{{\sf S}}

\newcommand{\mW}{{\sf W}}

\newcommand{\tautoolbox}{\texttt{Tau\xspace Toolbox \xspace}}
\newcommand{\fred}{\texttt{fred \xspace}}
\newcommand{\volt}{\texttt{volt \xspace}}
\newcommand{\orthval}{\texttt{orthoval \xspace}}
\newcommand{\polyval}{\texttt{polyval \xspace}}
\newcommand{\conv}{\texttt{conv}}
\newcommand{\matlab}{\textsc{Matlab}\xspace}
\newtheorem{theorem}{Theorem}
\newtheorem{lemma}[theorem]{Lemma}
\newtheorem{proposition}[theorem]{Proposition}
\newtheorem{example}[theorem]{Example}

\usepackage{geometry}
\allowdisplaybreaks

\begin{document}

\title{Spectral Lanczos' tau method for systems of nonlinear integro-differential equations\thanks{This work was partially supported by CMUP (UID/MAT/00144/2013), which is funded by FCT (Portugal) with national (MEC) and European structural funds (FEDER) under the partnership agreement PT2020.}}
\author{P. B. Vasconcelos\thanks{Center of Mathematics at University of Porto \& Economics Faculty at University of Porto, Portugal (pjv@fep.up.pt)} \and J. Matos\thanks{Center of Mathematics at University of Porto \& Politechnic School of Engeneering Porto, Portugal (jma@isep.ipp.pt} \and M. S. Trindade\thanks{Center of Mathematics at University of Porto, Portugal (marcelo.trindade.fc.up.pt)}}
\date{}
\maketitle

\abstract{In this paper an extension of the spectral Lanczos' tau method to systems of nonlinear integro-differential equations is proposed. This extension includes (i) linearization coefficients of orthogonal polynomials products issued from nonlinear terms and (ii) recursive relations to implement matrix inversion whenever a polynomial change of basis is required and (iii) orthogonal polynomial evaluations directly on the orthogonal basis. All these improvements ensure numerical stability and accuracy in the approximate solution. 
Exposed in detail, this novel approach is able to significantly outperform numerical approximations with other methods as well as different tau implementations. Numerical results on a set of problems illustrate the impact of the mathematical techniques introduced.}
\smallskip



\section{Introduction}\label{sec:intro}

The tau method is a spectral method, originally developed by
Lanczos in the 30's \cite{La38} that delivers polynomial approximations to the solution of differential problems. The method tackles both initial and boundary value problems with ease. It is a spectral method thus ensuring excellent error properties, whenever the solution is smooth.  

Initially developed for linear differential problems with polynomial coefficients, it has been used to solve broader mathematical formulations: functional coefficients, nonlinear differential and integro-differential equations. Several studies applying the tau method have been performed to approximate the solution of differential linear and non-linear equations \cite{CrRu83, LiPa99}, partial differential equations \cite{MaRoPa04, OrDi87} and integro-differential equations \cite{AbTa09, MaArSh05}, among others. Nevertheless, in all these works the tau method is tuned for the approximation of specific problems and not offered as a general purpose numerical tool.
   
A barrier to use the method as a general purpose technique has been the lack of automatic mechanisms to translate the integro-differential problem by an algebraic one. Furthermore and most importantly, problems often require high-order polynomial approximations, which brings numerical instability issues. The tau method inherits numerical instabilities from the large condition number associated with large matrices representing algebraically the actions of the integral, differential or integro-differential operator on the coefficients of the series solution. 

In this work numerical instabilities related with high-order polynomial approximations, in the tau method, are tackled allowing for the deployment of a general framework to solve integro-differential problems. We aim at contributing to broadcast the tau method for the scientific community and industry as it provides polynomials solutions with good error properties.
The \tautoolbox \cite{TrMaVa16} is a \matlab tool to solve integro-differential problems by the tau method. It aggregates all contributions available, enhances the use of the method by developing more stable algorithms and offers efficient implementations.  

\section{Preliminaries}\label{sec:preliminaries}

We begin by introducing the notation for the algebraic formulation of the tau method. 

Assume throughout that $\cP=[P_0,P_1,\ldots]\subseteq \bP$ is an orthogonal basis for the polynomials space $\bP$ of any non-negative integer degree, $\cX=[1,x,x^2,\ldots]\in\bP$ the power basis for $\bP$. Furthermore, consider that $y(x)=\sum_{i\geq 0} a_iP_i = \cP \msa$ is a formal series with coefficients $\msa=[a_0,a_1,\ldots]^T$. For the power basis, $y(x)=\sum_{i\geq 0} a_ix^i = \cX \msa_{\cX}$.

Lemma \ref{lemma:basischange} illustrates matrices $\mM$, $\mN$ and $\mO$ that set, respectively, polynomial multiplication, differentiation and integration into algebraic operations.

\begin{lemma}\label{lemma:basischange}
Let $\mV$ be the triangular matrix such that $\cP=\cX\mV$ and $\msa = \mV^{-1}\msa_{\cX}$. Then $x y = \cP \mM \msa$, $\ \frac{d}{dx} y = \cP \mN \msa$ and $\int ydx = \cP \mO \msa$ where  
\begin{equation}\label{matM}
\mM=\mV^{-1}\mM_{\cX} \mV, 
\quad \mM_{\cX} =
\begin{bmatrix}
0 &   &   &        & \\
1 & 0 &   &        & \\
  & 1 & 0 &        & \\
  &   & \ddots & \ddots
\end{bmatrix},
\end{equation}
\begin{equation}\label{matN}
\mN=\mV^{-1} \mN_{\cX} \mV, 
\quad \mN_{\cX} =
\begin{bmatrix}
0 & 1   &   &         &        \\
& 0 & 2 &   &         &        \\
&   & 0 & 3 &         &        \\
&   &   &  \ddots  & \ddots
\end{bmatrix} 
\end{equation}
and
\begin{equation}\label{matO}
\mO=\mV^{-1} \mO_{\cX} \mV,
\quad \mO_{\cX}=
\begin{bmatrix}
0 &             &             &        & \\
1 & 0           &             &        & \\ 
  & \frac{1}{2} & 0           &        & \\
  &             & \ddots & \ddots
\end{bmatrix}. 
\end{equation}

\end{lemma}
\begin{proof}
See \cite{OrSa81} for $\mM$ and $\mN$. It is then easy to extend to $\mO$ (see \cite{HoSh03}).
\end{proof}

The next proposition shows how to translate a linear ordinary differential and integral operators, with polynomial coefficients, into an algebraic representation. 
\begin{proposition} \label{prop:DopEop}
The $\nu$th order, $\nu\in\bN$, ordinary linear differential operator 
$\cD y=\sum_{k=0}^{\nu}p_k\frac{d^k y}{dx^k}$
and the $\gamma$th order, $\gamma\in\bN$, ordinary linear integral operator
$\cS y=\sum_{\ell=0}^{\gamma}p_\ell\left(\int y dx\right)^\ell$
acting on $\bP$, are casted on $\cP$ by, respectively,
\begin{equation}\label{Dopy}
\cD y=\cP \mD \msa,\quad \mD = \sum_{k=0}^{\nu}p_k(\mM) \mN^k 
\end{equation}
and 
\begin{equation}\label{Eopy}
\cS y=\cP \mS \msa,\quad \mS = \sum_{\ell=0}^{\gamma}p_\ell(\mM) \mO^\ell,
\end{equation}
with $p_r(\mM)=\sum_{i=0}^{n_r}p_{r,i}M^i$, $r=k,\ell$ and $n_r \in \mathbb{N}_0$.
\end{proposition}
\begin{proof}
Note that 
\begin{equation*}\label{MN}
\cX\mM_{\cX} \msa_{\cX}=x y,\quad \cX\mN_{\cX} \msa_{\cX}=\frac{d}{dx}y \quad \text{ and } \quad \cX \mO_{\cX} \msa_{\cX}=\int ydx
\end{equation*}
for
$\mM_{\cX}$, $\mN_{\cX}$ and $\mO_{\cX}$ as in Lemma \ref{lemma:basischange}. Then,  
\begin{enumerate}
\item[(i)] $xy=\cX \mM_{\cX} \msa_{\cX}=\cX \mV \mM \mV^{-1} \msa_{\cX}=\cP \mM \msa;$
\item[(ii)] $\frac{d}{dx}y = \cX \mN_{\cX}\msa_{\cX}=\cX \mV \mN \mV^{-1} \msa_{\cX}=\cP \mN \msa$ 
and thus 
$\cD y=\sum_{k=0}^{\nu}p_k\frac{d^k y}{dx^k}=\cP \mD \msa;$
\item[(iii)] $\int ydx = \cX \mO_{\cX} \msa_{\cX}=\cX \mV \mO \mV^{-1} \msa_{\cX}= \cP \mO \msa$ 
and thus
$\cS y=\sum_{\ell=0}^{\gamma}p_\ell\left(\int y dx\right)^\ell=\cP \mS \msa.$
\end{enumerate}
\end{proof}

Let $K(x,t)$ be a two-variable polynomial, or a two-variable polynomial approximation of a two- variable function. Then, for $\cP\big|_{x=t}=[P_0(t),P_1(t),\ldots]$ and $\mK\in\mathbb{R}^{n_x\times n_t}$,  
\begin{equation}\label{kernel}
K(x,t)=\sum_{i=0}^{n_x}\sum_{j=0}^{n_t}k_{i,j}P_i(x)P_j(t)=\cP \mK \cP^T \big|_{x=t}.
\end{equation}

\begin{lemma}\label{lemmaint}
For the integral operator $\int^{x}K(x,t)y(t)dt$, where $\int^{x}$ stands for the calculation of the integral at $x$, $K(x,t)=\sum_{i=0}^{n_x}\sum_{j=0}^{n_t}k_{ij}P_i(x)P_j(t)$ and $y(x)=\cP\msa$, one has
$\int^{x}K(x,t)y(t)dt=\cP \mS \msa,$
where 
$\mS=\sum_{i=0}^{n_x}\sum_{j=0}^{n_t}k_{ij}P_i(\mM)\mO P_j(\mM).$
\end{lemma}
\begin{proof}
Using Lemma \ref{lemma:basischange} one can deduce that
\begin{eqnarray*}
\label{intBONP}
\int^{x}x^it^jy(t)dt&=&\cX(\mV\mM^i\mV^{-1})(\mV\mO\mV^{-1})(\mV\mM^j\mV^{-1})\mV\msa\nonumber \\
&=&\cP\mM^i\mO\mM^j\msa,
\end{eqnarray*}
and therefore
\begin{equation*}\label{intBONO}
\int^{x}P_i(x)P_j(t)y(t)dt=\cP P_i(\mM)\mO P_j(\mM)\msa.
\end{equation*}
\end{proof}

\begin{lemma}\label{lemmavolterra}
$\int_{x_0}^{x}P_i(x)P_j(t)y(t)dt=\cP(P_i(\mM) -\mse_{i+1}\cP\big|_{x=x_0})\mO P_j(\mM)\msa,$
where $\mse_{i+1}$ is the $(i+1)$th column of the identity matrix.
\end{lemma}
\begin{proof}
From Lemma \ref{lemmaint}
	\begin{eqnarray}\label{}
	\int_{x_0}^{x}P_i(x)P_j(t)y(t)dt&=&\cP P_i(\mM) \mO P_j(\mM) \msa - \cP\big|_{x=x_0} P_i(\mM) \mO P_j(\mM) \msa \nonumber\\ 
	&=&\cP(P_i(\mM) -\mse_{i+1}\cP\big|_{x=x_0})\mO P_j(\mM)\msa. \nonumber 
	\end{eqnarray}
	Note that it is easy to understand that $\cX\big|_{x=x_0}\mM_\cX^i=\mse_{i+1}\cX\big|_{x=x_0}$.
\end{proof}

\begin{lemma}\label{lemmafredholm}
$\int_{a}^{b}P_i(x)P_j(t)y(t)dt=\cP \mse_{i+1}(\cP\big|_{x=b} -\cP\big|_{x=a})\mO P_j(\mM)\msa.$
\end{lemma}
\begin{proof}
Immediate since it is a particular case of Lemma \ref{lemmavolterra}. 
\end{proof}

\begin{proposition}
The linear Volterra integral operator $\cS_V y = \int_{x_0}^{x}K(x,t)y(t)dt$ and the Fredholm integral operator $\cS_F y = \int_{a}^{b}K(x,t)y(t)dt$, with degenerate kernel $K(x,t)\approx\sum_{i=0}^{n_x}\sum_{j=0}^{n_t}k_{ij}P_i(x)P_j(t)$,  smooth and continuous, acting on $y$ have, respectively, the following algebraic representation
\begin{equation}\label{orth_volt}
\mS_V y=\sum_{i=0}^{n_x}\sum_{j=0}^{n_t}k_{ij}\left(P_i(\mM)-\mse_{i+1}\cP\big|_{x=x_0}\right)\mO P_j(\mM)\msa
\end{equation}
and
\begin{equation}\label{orth_fred}
\mS_F y=\sum_{i=0}^{n_x}\sum_{j=0}^{n_t}k_{ij}\mse_{i+1}\Big(\cP\big|_{x=b}-\cP\big|_{x=a}\Big)\mO P_j(\mM)\msa
\end{equation}
for $\mse_i$ the $i$th column of the identity matrix.
\end{proposition}
\begin{proof}
Equation \eqref{orth_volt} can be immediately obtained from Lemma \ref{lemmavolterra} (see e.g. \cite{SaTaMa13}) and equation \eqref{orth_fred} from Lemma \ref{lemmafredholm} (see e.g. \cite{HoSh03}).
\end{proof}

\section{The tau method for integro-differential problems}\label{sec:tau}

An approximate polynomial solution $y_n$ for the linear integro-differential problem 
\begin{equation}
  \begin{cases}\label{Gy=f}
  \cD y+\cS y+\cS_V y+\cS_F y=f \\
  c_i(y)=s_i,\, i=1,\ldots,\nu
  \end{cases},
\end{equation}
is obtained in the tau sense by solving a perturbed system 
\begin{equation}
  \begin{cases}\label{Gyn=f+taun}
  \cD y_n+\cS y_n+\cS_V y_n+\cS_F y_n=f+\tau_n \\
  c_i(y_n)=s_i,\, i=1,\ldots,\nu
  \end{cases},
\end{equation}
where $f$ is a $\lambda$th degree polynomial (or a polynomial approximation of a function), $\tau_n$ is the residual and $c_i(y)=s_i$, $i=1,\ldots,\nu$ the initial and/or boundary conditions. 

Problem \eqref{Gy=f} has a matrix representation given by
\begin{equation}\label{Ta=b}
  \begin{cases}
  \mC \msa=\mss \\
  (\mD+\mS+\mS_V+\mS_F) \msa=\msf
  \end{cases},
\end{equation}
where 
$\mC=[c_{ij}]_{\nu\times\infty},\quad c_{ij}=c_i(P_{j-1}),\, i=1,\ldots,\nu,\, j=1,2,\ldots$, 
$\msa=[a_0,a_1,\ldots]^T$ the coefficients of $y$ in $\cP$, $\mss=[s_1,\ldots,s_\nu]^T$, 	
$\mD$, $\mS$, $\mS_V$ and $\mS_F$ as defined in, respectively, \eqref{Dopy}-\eqref{orth_fred}, and $\msf=[f_0,\ldots,f_\lambda,0,0,\ldots]^T$ the right hand side of the system.

Choosing an integer $n\geq \nu+\lambda$, an $(n-1)$th degree polynomial approximate solution $y_n=\cP_n \msa_n$ is obtained by truncating system \eqref{Ta=b} to its first $n$ columns. Moreover, restricting this system to its first $n+\nu+h$ equations, a linear system
of dimension $n\times n$ is obtained, which is equivalent to introduce a polynomial residual 
\begin{equation}
\label{eq:taun_analitic}
\tau_n= (\cD y - \cD y_n) + (\cS y- \cS y_n) + (\cS_V y- \cS_V y_n) + (\cS_F y- \cS_F y_n). 
\end{equation}

\section{Nonlinear approach for integro-differential problems}\label{sec:nonlinear}
Nonlinear differential problems are tackled with linear approximations and solving a set of linear problems. 

Let $G$ be the nonlinear operator acting on an appropriate space of smooth functions
\begin{equation} \label{Geq0}
G(y^{(-\gamma)},\ldots,y^{(-1)},y^{(0)},y^{(1)},\ldots,y^{(\nu)})=0,
\end{equation}
where $\displaystyle y^{(\ell)}=\frac{d^{\ell}y}{dx^{\ell}}$ for $\ell\in \mathbb{Z}^{+}$, $\displaystyle y^{(\ell)}=(\int ydt)^{\ell}$ for $\ell\in \mathbb{Z}^{-}$ and includes Volterra and Fredholm terms.
	
If $G$ is $\mathbb{C}^1$ in a neighborhood $\Omega$ of $\omega=(y^{(-\gamma)},\ldots,y^{(-1)},y^{(0)},y^{(1)},\ldots,y^{(\nu)})$ and if $\omega_0 \in \Omega$ is an approximation of $\omega$, then a linear operator $T$ can be defined, represented by the order one Taylor polynomial centered at $\omega_0$ 
\begin{equation*} \label{T}
T(\omega)=G(\omega_0)+\sum_{i=-\gamma}^{\nu} \frac{\partial G}{\partial y^{(i)}}|_{\omega_0}(y^{(i)}-y_0^{(i)})
\end{equation*}

As in the Newton method for algebraic equations, we can replace $G$ by $T$ in \eqref{Geq0} and solve the approximated equation
\begin{equation} \label{Teq0}
\sum_{i=-\gamma}^{\nu} \frac{\partial G}{\partial y^{(i)}}|_{\omega_0} y^{(i)} = -G(\omega_0) + \sum_{i=-\gamma}^{\nu} \frac{\partial G}{\partial y^{(i)}}|_{\omega_0} y_0^{(i)}.
\end{equation}

Applying the Tau method to the linear differential equation \eqref{Teq0} and taking $\omega_1=(y_1^{(-\gamma)},\ldots,y_1^{(\nu)})$ as the solution, if $\omega_1\in \Omega$ we can repeat the process, obtaining an iterative procedure, solving for $\omega_k$ the linear differential equation
\begin{equation*} \label{Tk}
\sum_{i=-\gamma}^{\nu} \frac{\partial G}{\partial y^{(i)}}|_{\omega_{k-1}} y_k^{(i)} = - G(\omega_{k-1}) + \sum_{i=-\gamma}^{\nu} \frac{\partial G}{\partial y^{(i)}}|_{\omega_{k-1}} y_{k-1}^{(i)},\quad k=1,2,\ldots.
\end{equation*}

\section{Contributions to stability}\label{sec:contributions}

In this section we summarize some of the mathematical techniques developed for the tau method to provide stable algorithms for the \tautoolbox library. 

Let $\cP=[P_0(x),P_1(x),\ldots]$ be an orthogonal basis satisfying $xP_j=\alpha_jP_{j+1}+\beta_jP_j+\gamma_jP_{j-1},\, j\geq 0,\, P_0=1,\, P_{-1}=0$.

\hfill \newline
\noindent \textbf{Orthogonal evaluation:}
If $\cP^*$ are the corresponding orthogonal polynomials shifted to $[a,\ b]$ and $x$ is a vector, then the evaluation of $y_n(x)=\sum_{i=0}^n a_iP^*_i$ is directly computed in $\cP^*$ by the recursive relation
\begin{equation*}
\begin{cases}
P^*_0=[1,\ldots,1]^T \\
P^*_1=\frac{c_1 x + (c_2 -\beta_0)P^*_0}{\alpha_0} \\
P^*_i=\frac{(c_1 x+\msc_2-\beta_{i-1}P^*_0)\odot P^*_{i-1}-\gamma_{i-1}P^*_{i-2}}{\alpha_{i-1}},\quad i=2,3,\ldots,n
\end{cases}
\end{equation*}
where $\odot$ is the element-wise product of two vectors, $c_1 = \frac{2}{b-a}$ and $c_2=\frac{a+b}{a-b}$.

The \tautoolbox proposes a \orthval function, instead of the \polyval \matlab one, to implement this functionality. 

\hfill \newline
\noindent \textbf{Change of basis by recurrence:}
Let $\mV$ satisfy $\msa_\cX=\mV\msa$. The coefficients of $\mW=\mV^{-1}$ can be computed without inverting $\mV$ by the recurrence relation
\begin{equation}\label{propW}
	\begin{cases}
	\msw_{1}=\mse_1 \\
	\msw_{j+1}=\mM\msw_{j},\quad j=1,2,...
	\end{cases},
\end{equation}
where $\mM$ is such that $\cP x=\cP \mM$, $\msw_j$ is the $j$th column of $\mM$ and $\mse_1$ the first column of the identity matrix.

\hfill \newline
\noindent \textbf{Avoiding similarity transformations:}
Matrix inversion presented at all similarity transformations must be avoided to ensure numerically stable computations. Recurrence relations to compute the elements of matrices $\mM$, $\mN$ and $\mO$ directly on $\cP$ can be computed, respectively, by 
	\begin{equation*}\label{M}
	\mM=\left[\mu_{i,j}\right]_{i,j=1}^{n}=
	\begin{cases}
	\mu_{j+1,j}=\alpha_{j-1},\quad \mu_{j,j}=\beta_{j-1}, \quad \mu_{j-1,j}=\gamma_{j-1}\\
	\mu_{i,j}=0,\quad |i-j|>1 
	\end{cases},
	\end{equation*}
	\begin{equation*}\label{N}
	\mN=\left[\eta_{i,j}\right]_{i,j=1}^{n}=
	\begin{cases}
	\eta_{i,j+1}=\frac{\alpha_{i-1}\eta_{j,i-1}+(\beta_i-\beta_j)\eta_{j,i}+\gamma_{i+1}\eta_{j,i+1}-\gamma_j\eta_{j-1,i}}{\alpha_j}\\
	\eta_{j,j+1}=\frac{\alpha_{j-1}\eta_{j,j-1}+1}{\alpha_j}\\
	\end{cases}, \text{ and }		
	\end{equation*}
	\begin{equation*}\label{O}
	\mO=\left[\theta_{i,j}\right]_{i,j=1}^{n}=
	\begin{cases}
	\theta_{j+1,j} = \frac{\alpha_j}{j+1}\\
	\theta_{i+1,j}=-\frac{a_i}{i+1}\sum_{k=i+2}^{j+1}\eta_{i,k}\theta_{k,j},\quad i =j-1,\ldots,1,0
	\end{cases}.
	\end{equation*}

\hfill \newline
\noindent \textbf{Linearization coefficients:}
Product of polynomials $\msp$ and $\msq$ in $\cP$ occurs in nonlinear problems. 
Usually both polynomials are translated first from $\cP$ to $\cX$, then the convolution is applied and finally the product is translated back, $\msp \msq = \mV^{-1}\conv (\mV \msp, \mV \msq)$.
Alternatively, ensuring robustness, this product can be directly computed on $\cP$ using the linearization coefficients: $P_iP_j=\sum_{k=0}^{i+j} l_{i,j,k} P_k$, where $l_{i,j,k}$ are computed by recurrence relations \cite{Le96}.

For $\msp=\cP \msa$ and $\msq=\cP \msb$, of degree $n$, then $\msp \msq=\cP \msc$ where
\[ \begin{cases}		
c_k = {\displaystyle \sum_{j=0}^{k} \sum_{i=\lfloor\frac{k+1-j}{2}\rfloor}^{n-j} (\frac{1}{2})^{\delta_{i,j}}(a_ib_j+a_jb_i)l_{i,i+j,k}}, \quad  k=0,\ldots,n\\
c_{n+k} = {\displaystyle \sum_{j=0}^{n-k} \sum_{i=\lfloor\frac{n-k-j}{2}\rfloor}^{n-j} (\frac{1}{2})^{\delta_{i,j}}(a_ib_j+a_jb_i)l_{i,i+j,n+k}}, \quad  k=1,\ldots,n		
\end{cases}. \]

\hfill \newline
\noindent \textbf{Computing with $\mM$:} An efficient way to compute the powers of $M$ is 
$$\mM^{k} = [\mu_{i,j}^{(k)}],\quad \mu_{i,j}^{(k)} = \mu_{i-1,j}^{(k-1)}\alpha_{i-1}+\mu_{i,j}^{(k-1)}\beta_i+\mu_{i+1,j}^{(k-1)}\gamma_{i+1}.$$ 

Moreover, the evaluation of $y_n(\mM)=\sum_{i=0}^n a_{n,i} P_i(\mM)$ can be performed with $$P_{j+1}(\mM)=\frac{(\mM - \beta_j I)P_j(\mM) - \gamma_jP_{j-1}(\mM)}{\alpha_j},\quad j\geq 0,$$
 $P_0(\mM)=I$ and $P_{-1}(\mM)=\varnothing$.

\section{Numerical results}\label{sec:results}
\begin{example}
Consider the integro-differential Fredholm nonlinear equation of the second kind \cite{DeSa12}, with exact solution $y(x)=\exp(-x)$,
	\begin{equation}
	\displaystyle\frac{d}{dx}y(x)+y(x)-\int_{0}^{1}y(t)^2dt=0.5(e^{-2}-1),\quad y(0) =1.
	\label{non_lin_ideF}
	\end{equation}	
\end{example}
Introducing new variables $y_1 = y$ and $y_2 = y_1^2$, we get $\frac{d}{dx}y_2 = 2y_1\frac{d}{dx}y_1$. Linearizing it results $ y_1\frac{d}{dx}y_1\approx y_1^{(k)}\frac{d}{dx}y_1 + \frac{d}{dx}y_1^{(k)}y_1 -y_1^{(k)}\frac{d}{dx}y_1^{(k)}$ and, therefore, problem \eqref{non_lin_ideF} can be casted as
	\begin{equation}
	\begin{cases}
	\displaystyle\frac{d}{dx}y_1^{(k+1)}+y_1^{(k+1)}-\int_{0}^{1}y_2^{(k+1)}dt=0.5(e^{-2}-1)\\
	\displaystyle\frac{d}{dx}y_2^{(k+1)}-2\left(y_1^{(k)}\frac{d}{dx}y_1^{(k+1)}+y_1^{(k+1)}\frac{d}{dx}y_1^{(k)}\right)=-2y_1^{(k)}\frac{d}{dx}y_1^{(k)}\\
	y_1^{(1)}(0)=1,\quad y_2^{(1)}(0)=1
	\end{cases},
	\end{equation}
	and the coefficient matrix (ChebyshevT basis) of the linear system $\mT_\cT\msa_\mathcal{T}=\msb_\cT$ is     
	\begin{equation*}{\sf T}_\mathcal{T}=
	\begin{bmatrix}
	T_0(0)\quad \ldots\quad T_{n-1}(0) & \bf 0 \\
	\bf 0 & T_0(0)\quad \ldots\quad T_{n-1}(0) \\
	\mM_\cT+\sf I& -\mS_{F_\cT} \\
	-2\left[y_1^{(k)}(\mM_\cT)\mN_\cT+{y_1^{(k)}}'(\mM_\cT){\sf I}\right]&\mN_\cT \\
	\end{bmatrix},
	\end{equation*}
	where $\mS_{F_\cT}$ by \eqref{orth_fred}, making use of the \tautoolbox \fred function. The independent vector is $\msb_\cT=\left[1,\ 1,\ 0.5(e^{-2}-1),\ 0,\ \ldots,\ 0,\ -2y_1^{(k)}\frac{d}{dx}y_1^{(k)}\right]^T.$
	
The error presented in \cite{DeSa12} is $||\mse||_\infty=\max_j|y_n(x_j)-y(x_j)|$, and therefore the same measure is applied to compare the results, see Table \ref{aVandW}.
	\begin{table}[htb]
		\centering
		\begin{tabular*}{0.75\textwidth}{@{\extracolsep{\fill}} r r r r r}			\hline
			& \multicolumn{2}{c}{$\cite{DeSa12}$} & \multicolumn{2}{c}{\tautoolbox} \\
			$n$& $||{\sf e}||_{\infty}$ & CPU time & $||{\sf e}||_{\infty}$ & CPU time \\
			\hline
			5&  	9.63e-04&    0.42&   1.58e-04& 	0.03\\
			9& 	1.28e-04&    0.58&   1.28e-09& 	0.03\\
			17& 	2.87e-05&    0.73&   7.77e-16& 	0.04\\
			33& 	5.61e-06&    0.07&   4.44e-16& 	0.07\\
			65& 	2.39e-06&    1.54&   4.44e-16& 	0.37\\
			129& 	1.28e-06&    2.15&   4.44e-16& 	2.35\\
			\hline
		\end{tabular*}
		\caption{Comparison between the results at \cite{DeSa12} and with \tautoolbox.}
		\label{aVandW}
	\end{table}

For all polynomial degree approximations the approximate solution provided by \tautoolbox is clearly better than the one given by \cite{DeSa12}. The error for $n=129$ with \cite{DeSa12} was reached with \tautoolbox with a polynomial degree smaller than $9$ and for degree $17$ the  \tautoolbox was able to find the solution with machine precision order. Noteworthy is that for increasing polynomial degree the algorithm is stable, not showing any perturbations for higher degrees. The CPU time should not be compared between both approaches since results are reported from two distinct machines. It is however relevant to understand that the effort to solve the problem with \tautoolbox is higher: if we take $n=5$ as reference time then for $n=128$ \tautoolbox required $78.3\times$ the reference computational cost whereas \cite{DeSa12} necessitates $5.12\times$. Robustness and stability comes at a price: more elaborate mathematics must be performed to reach such quality results. Nevertheless, one must point out that the CPU time to compute an accurate approximation was still low, only $2.35$ seconds. 

\begin{example}
Consider now the system of integro-differential equations with nonlinear Volterra term \cite{AbTa09}, with exact solution $y_1(x)=\sinh(x)$ and $y_2(x)=\cosh(x)$,
	\begin{equation}
	\begin{cases}
	\displaystyle \frac{d}{dx}y_1(x)+\frac{1}{2}{{\left(\frac{d}{dx}y_2(x)\right)^2}}-\int_{0}^{x}(x-t)y_2(t)+{{y_2(t)y_1(t)}}dt=1\\
	\displaystyle
	\frac{d}{dx}y_2(x)-\int_{0}^{x}(x-t)y_1(t)-{{y_2^2(t)}}+{{y_1^2(t)}}dt=2x\\
	\displaystyle
	y_1(0)=0,\quad y_2(0)=1
	\end{cases}.
	\label{non_lin_ideV}
	\end{equation}
\end{example}

As for the previous example, linearization is done first and the Volterra integral term $\mS_V$, following \eqref{orth_volt}, is tackled using the \tautoolbox \volt function.

Fig. \ref{fig:example2} shows the error after 5 iterations along the interval $[0,1]$. For $n=20$, \tautoolbox was able to provide an approximate solution with machine precision all over the interval. For comparison purposes, results for the same problem in \cite{AbTa09} with $n=10$ are plotted together with \tautoolbox library. The former, at the right part of the interval, can only reach single-precision accuracy, in contrast with the latter, which delivers double precision accuracy. For $n\ge25$, \tautoolbox reaches machine precision.  

\begin{figure}[H]
\centering
\includegraphics[width=1.0\textwidth]{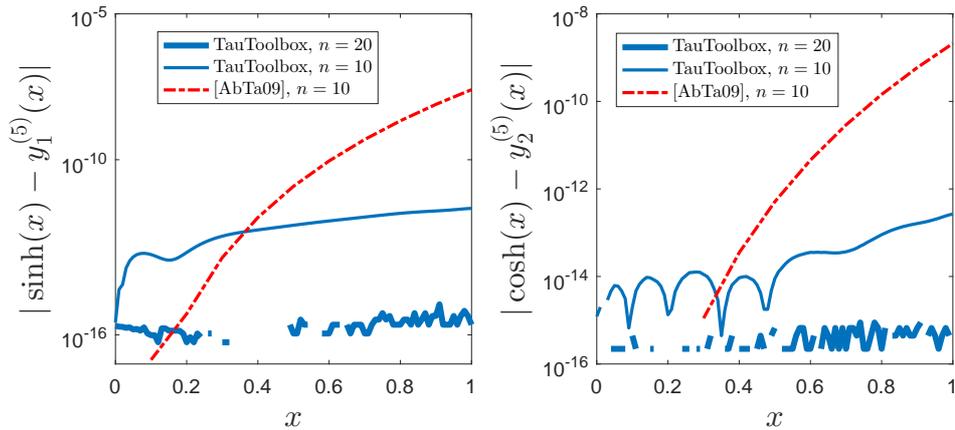}
\caption{Comparison between \cite{AbTa09} and \tautoolbox for $10$ (and $20$) degree polynomial.}
\label{fig:example2}
\end{figure}

\section{Conclusions}\label{sec:conclusion}
In this work we proposed the Lanczos' tau method for nonlinear integro-differential systems of equations. Contributions to improve the stability of the numerical implementation are presented.  Numerical experiments illustrate the accuracy and efficiency of the new proposal, when compared with the results in the literature. 

All these contributions are included at \tautoolbox -- a \matlab library for the solution of integro-differential problems.



\begin{thebibliography}{10}

\bibitem{AbTa09}
S.~Abbasbandy and A.~Taati.
\newblock Numerical solution of the system of nonlinear volterra
  integro-differential equations with nonlinear differential part by the
  operational tau method and error estimation.
\newblock {\em Journal of Computational and Applied Mathematics}, 231(1):106 --
  113, 2009.


\bibitem{CrRu83}
M.~R. Crisci and E.~Russo.
\newblock An extension of ortiz' recursive formulation of the tau method to
  certain linear systems of ordinary differential equations.
\newblock {\em Mathematics of Computation}, 41(163):27--42, 1983.

\bibitem{DeSa12}
M.~Dehghan and R.~Salehi.
\newblock The numerical solution of the non-linear integro-differential
  equations based on the meshless method.
\newblock {\em Journal of Computational and Applied Mathematics}, 236(9):2367
  -- 2377, 2012.

\bibitem{HoSh03}
S.~Hosseini and S.~Shahmorad.
\newblock Numerical solution of a class of integro-differential equations by
  the tau method with an error estimation.
\newblock {\em Applied Mathematics and Computation}, 136(2--3):559 -- 570,
  2003.

\bibitem{La38}
C.~Lanczos.
\newblock Trigonometric interpolation of empirical and analytical functions.
\newblock {\em Studies in Applied Mathematics}, 17(1-4):123--199, 1938.

\bibitem{Le96}
S.~Lewanowicz.
\newblock Second-order recurrence relation for the linearization coefficients
  of the classical orthogonal polynomials.
\newblock {\em Journal of Computational and Applied Mathematics}, 69(1):159 --
  170, 1996.

\bibitem{LiPa99}
K.~Liu and C.~Pan.
\newblock The automatic solution to systems of ordinary differential equations
  by the tau method.
\newblock {\em Computers \& Mathematics with Applications}, 38(9--10):197 --
  210, 1999.

\bibitem{MaRoPa04}
J.~Matos, M.~J. Rodrigues, and P.~B. Vasconcelos.
\newblock New implementation of the tau method for \{PDEs\}.
\newblock {\em Journal of Computational and Applied Mathematics}, 164--165:555
  -- 567, 2004.

\bibitem{OrDi87}
E.~L. Ortiz and A.~P.~N. Dinh.
\newblock Linear recursive schemes associated with some nonlinear partial
  differential equations in one dimension and the tau method.
\newblock {\em SIAM Journal on Mathematical Analysis}, 18(2):452--464, 1987.

\bibitem{OrSa81}
E.~L. Ortiz and H.~Samara.
\newblock An operational approach to the tau method for the numerical solution
  of non-linear differential equations.
\newblock {\em Computing}, 27(1):15--25, 1981.

\bibitem{MaArSh05}
J.~Pour-Mahmoud, M.~Y. Rahimi-Ardabili, and S.~Shahmorad.
\newblock Numerical solution of the system of fredholm integro-differential
  equations by the tau method.
\newblock {\em Applied Mathematics and Computation}, 168(1):465--478, Sept. 2005.

\bibitem{SaTaMa13}
L.~Saeedi, A.~Tari, and S.~H.~M. Masuleh.
\newblock Numerical solution of some nonlinear volterra integral equations of
  the first kind.
\newblock {\em Applications and Applied Mathematics}, 8(1):214--216, 2013.

\bibitem{TrMaVa16}
M.~Trindade, J.~Matos, and P.~B. Vasconcelos.
\newblock Towards a lanczos' $\tau$-method toolkit for differential problems.
\newblock {\em Mathematics in Computer Science}, 10(3):313--329, 2016.

\end{thebibliography}

\end{document}